\documentclass[11pt]{article}

\usepackage{graphicx}
\usepackage{latexsym}
\usepackage[usenames,dvipsnames]{color}
\usepackage{amssymb,amsmath}
\usepackage{lscape}
\usepackage{verbatim}
\usepackage{bm}

\normalbaselines
\catcode `\@=11
\@addtoreset{equation}{section}

\def\section{\@startsection {section}{1}{\z@}{-3.5ex plus -1ex minus -.2ex}{2.3ex plus .2ex}{\large\bf}}
\def\subsection{ \@startsection{subsection}{2}{\z@}{-3.25ex plus -1ex minus -.2ex}{1.5ex plus .2ex}{\normalsize \bf}}
\def\subsubsection{\@startsection{subsubsection}{3}{\z@}{-3.25ex plus -1ex minus -.2ex}{1.5ex plus .2ex}{\normalsize\sl}}
\catcode`\@=12

\newtheorem{theorem}{Theorem}[section]
\newtheorem{defn}[theorem]{Definition}
\newtheorem{prop}[theorem]{Proposition}
\newtheorem{lemma}[theorem]{Lemma}
\newtheorem{cor}[theorem]{Corollary}

\newtheorem{rems}[theorem]{Remarks}
\newtheorem{rem}[theorem]{Remark}
\newtheorem{example}[theorem]{Example}

\newenvironment{proof}{\noindent \bf Proof : \rm}{$ \hspace{\stretch{1}} \Box $}

 \newcommand{\bea}{\begin{eqnarray}}
\newcommand{\ena}{\end{eqnarray}}

\newcommand{\beano}{\begin{eqnarray*}}
\newcommand{\enano}{\end{eqnarray*}}

\newcommand{\bei}{\begin{itemize}}
\newcommand{\eni}{\end{itemize}}

\newcommand{\be}{\begin{equation}}
\newcommand{\en}{\end{equation}}

\newcommand{\bedefi}{\begin{defn} \rm }
\newcommand{\findefi}{\end{defn}}

\newcommand{\belem}{\begin{lemma}}
\newcommand{\enlem}{\end{lemma}}

\newcommand{\beprop}{\begin{prop}}
\newcommand{\enprop}{\end{prop}}

\newcommand{\betheo}{\begin{theorem}}
\newcommand{\entheo}{\end{theorem}}

\newcommand{\becor}{\begin{cor}}
\newcommand{\encor}{\end{cor}}

\newcommand{\berem}{\begin{rem} \rm}
\newcommand{\enrem}{\end{rem}}

\newcommand{\berems}{\begin{rems} \rm}
\newcommand{\enrems}{\end{rems}}

\newcommand{\beex}{\begin{example}$\!\!\!$ \rm }
\newcommand{\enex}{ \end{example}}

\newcommand{\norm}[2]{
\left\| #2 \right\|_{#1}
}
\newcommand{\Hil}[0]{\mathcal{H} }

\renewcommand{\le}{\leqslant}
\renewcommand{\leq}{\leqslant}
\renewcommand{\geq}{\geqslant}

 \newcommand{\noi}{\noindent}

 \newcommand{\ov}{\overline}

\newcommand{\nN}{\mathbb{N}}
\newcommand{\RN}{\mathbb{R}}
\newcommand{\ZN}{\mathbb{Z}}
\newcommand{\CN}{{\mathbb C}}

\newcommand{\ran}{{\sf Ran}\,}
\newcommand{\Ker}{{\sf Ker}\,}
\newcommand{\hs}{Hilbert space}
\newcommand{\pip}{{\sc pip}-space}

\def\G{{\mathcal G}}

\def\H{{\mathcal H}}
\def\M{{\mathcal M}}
\def\V{{\mathcal V}}

\def\T{\widehat{T}}

\newcommand{\NN}[0]{\mathbb{N}}
\newcommand{\ud}{\,\mathrm{d}}

\newcommand{\cl}[2]{[ {#1}]_{#2}}

\newcommand{\ip}[2]{ \langle {#1} |{#2}  \rangle}
\newcommand{\ipp}[2]{ {\bm\langle}\!\!\!{\bm\langle} {#1} |{#2}{\bm\rangle}\!\!\!{\bm\rangle}_{\scriptscriptstyle\mu}}

\def\<{\langle}
\def\>{\rangle}

\definecolor{teal}{rgb}{0.0, 0.5, 0.5}

\usepackage{dsfont}

\begin{document}

\begin{flushleft}
{\Large \sc Reproducing pairs of measurable functions} \vspace*{7mm}

{\large\sf   J-P. Antoine $\!^{\rm a}$, M. Speckbacher  $\!^{\rm b}$
and C. Trapani $\!^{\rm c}$
}
\\[3mm]
$^{\rm a}$  {\small Institut de Recherche en Math\'ematique et  Physique, Universit\'e catholique de Louvain \\
\hspace*{3mm}B-1348   Louvain-la-Neuve, Belgium\\
\hspace*{3mm}{\it E-mail address}: jean-pierre.antoine@uclouvain.be}
\\[1mm]
$^{\rm b}$  {\small Acoustic Research Institute, Austrian Academy of Science\\
\hspace*{3mm}A-1040 Vienna, Austria\\
\hspace*{3mm} {\it E-mail address}: speckbacher@kfs.oeaw.ac.at}
\\[1mm]
$^{\rm c}$ {\small Dipartimento di Matematica e Informatica,
Universit\`a di Palermo, \\
\hspace*{3mm} I-90123 Palermo, Italy\\
\hspace*{3mm} {\it E-mail address}: camillo.trapani@unipa.it}
\end{flushleft}

 \begin{abstract}
We analyze the notion of reproducing pair of {weakly} measurable functions, which generalizes that of continuous frame. We show, in particular, that each reproducing pair generates  two    Hilbert  spaces,  conjugate dual  to each other.
Several examples, both discrete and continuous, are presented.
\medskip

\textbf{AMS classification numbers:} 41A99, 46Bxx,  46Exx
\medskip


\end{abstract}

\section{Introduction}
\label{sec-intro}

Frames and their relatives are most  often considered in the discrete case, for instance in signal processing \cite{christ}.
However, continuous frames have also been studied and offer interesting mathematical problems. They have been introduced originally by Ali, Gazeau and one of us
\cite{squ-int,contframes} and also, independently, by Kaiser \cite{kaiser}. Since then, several papers dealt with various aspects of the
concept, see for instance
\cite{gab-han} or \cite{rahimi}.
 {However, there may occur situations where it is impossible to satisfy both frame bounds.}

 {
Therefore, several generalizations of frames have been introduced. The concept of semi-frames
\cite{ant-bal-semiframes1,ant-bal-semiframes2}, for example, is concerned with functions that  only satisfy one of the two frame bounds.}
It turns out that a large portion of frame theory can be extended to this larger framework,
in particular the notion of duality.

More recently, a new generalization of frames was introduced by Balazs and one of us \cite{speck-bal}, namely, reproducing
pairs. Here one
considers a couple of weakly measurable functions $(\psi, \phi)$,  {instead of a single mapping},
and one studies what amounts to the correlation between the two (a precise definition is given below).
This definition also includes the original definition of a continuous frame \cite{squ-int,contframes} given the choice
$\psi= \phi$. Moreover,  it gives rise to a continuous and invertible analysis/synthesis process without the need  of any
frame bounds. The increase of  freedom in choosing the mappings $\psi$ and $\phi$, however, leads to the problem of characterizing
the range of the analysis operators.

We  {will show in Section \ref{sec-Banach} that this problem can be solved
by introducing a pair of  intrinsically generated Hilbert spaces, {conjugate dual} to each other.}
We discuss in detail the properties of these spaces,
in particular, we examine when a given function has a reproducing partner.
In Section \ref{sec-examples}, we exhibit several concrete examples of the construction, both in the discrete and in the
continuous cases.
In particular, we show that the wavelet upper semi-frame  described in \cite{ant-bal-semiframes1} does not admit a
 {second mapping to form a} reproducing pair.

\smallskip

\section{Preliminaries}\label{prel}
Before proceeding, we list our definitions and  conventions. The framework is
 a (separable) Hilbert space $\H$, with the inner product $\ip{\cdot}{\cdot}$ linear in the first factor.
Given an operator $A$ on $\H$, we denote its domain by ${\sf Dom}A$, its range by $\ran A$ and its kernel by ${\Ker} A$.
$GL(\H)$ denotes the set of all invertible bounded operators on $\H$ with bounded inverse.
Throughout the paper, we will consider weakly measurable functions $\psi: X \to \H$, where $(X,\mu)$ is a  locally compact  space with a Radon measure $\mu$.
Then the weakly measurable function $\psi$ is a \emph{continuous frame} if there exist constants  ${\sf m} > 0$  and ${\sf M}<\infty$  (the  frame bounds) such that
\be\label{eq:frame}
{\sf m}  \norm{}{f}^2 \leq    \int_{X}  |\ip{f}{\psi_{x}}| ^2 \, \ud \mu(x)  \le {\sf M}  \norm{}{f}^2 ,  \forall \, f \in \H.
\end{equation}
Given the continuous frame $\psi$, the  \emph{analysis} operator ${\sf C}_{\psi}: \Hil \to L^{2}(X, \ud\mu)$
\cite{foot1}
is defined as
\be\label{eq:csmap}
({\sf C}_{\psi}f)(x) =\ip{f} {\psi_{x}}, \; f \in \H,
\end{equation}
and the corresponding \emph{synthesis operator} ${\sf C}_{\psi}^\ast: L^{2}(X, \ud\mu) \to \H$ as
 (the integral being understood in the weak sense,  as usual)
\be\label{eq:synthmap}
{\sf C}_{\psi}^\ast \xi =  \int_X  \xi(x) \,\psi_{x} \; \ud\mu(x), \mbox{ for} \;\;\xi\in L^{2}(X, \ud\mu).
\end{equation}
We set   $S_\psi:={\sf C}_{\psi}^* {\sf C}_{\psi}$, {which is self-adjoint.}

Then it follows that
$$\ip{S_\psi f}{g}=\ip{{\sf C}^*_\psi {\sf C}_\psi f}{g} = \ip{ {\sf C}_\psi f}{{\sf C}_\psi g} = \int_{X}  \ip{f}{\psi_{x}} \ip{\psi_x}{g} \, \ud \mu(x) .$$ Thus, for continuous frames, $S_\psi$ and $S_\psi^{-1}$ are both bounded, that is, $S_\psi\in GL(\H)$.

 {The weakly measurable function $ {\psi}$ is said to be   \emph{$\mu$-total} if $\ip{\psi_x}{g} = 0$, a.e., implies $g=0$,
that is, $\Ker {\sf C}_{\phi} = \{0\}$.}

Now, in practice, there are situations where the notion of frame is too restrictive, in the sense that one cannot satisfy \emph{both}
frame bounds simultaneously.
Thus there is room for two natural generalizations. Following \cite{ant-bal-semiframes1,ant-bal-semiframes2}, we will say that
a family  $\psi$ is   an \emph{upper (resp. lower) semi-frame}, if
it is  {$\mu$-total} in $\H$ and  satisfies the upper (resp. lower) frame inequality.
For the sake of completeness, we recall the definitions.
A weakly measurable function $\psi$ is   an \emph{upper semi-frame} if {there} exists ${\sf M}<\infty$ such that
\be\label{eq:upframe}
0 < \int_{X}  |\ip{f}{\psi_{x}}| ^2 \, \ud \mu(x)   \leq { \sf M}  \norm{}{f}^2 , \; \forall \, f \in \H, \, f\neq 0 .
\en
Note that an  upper   semi-frame is also called a total Bessel mapping \cite{gab-han}.
On the other hand, a   function $ \psi$  is a \emph{lower semi-frame} if  there exists   a constant ${\sf m}>0$ such that
\be
{\sf m}  \norm{}{f}^2 \leq  \int_{X}  |\ip{f}{\psi_{x}}| ^2 \, \ud \mu(x) , \;\; \forall \, f \in \H.
\label{eq:lowersf}
\end{equation}
 Note that the lower frame inequality automatically implies  that the family is  {$\mu$-total.}
Thus, if $\psi$ is an upper semi-frame and not a frame, $S_\psi$   is bounded and $S_\psi^{-1}$ is unbounded,
{as follows immediately from \eqref{eq:upframe}.}

{In the lower case, however, the definition of $S_\psi$  must be changed, since ${\sf C}_{\psi}$ need not be densely defined, so that ${\sf C}_{\psi}^*$ may not exist. Instead, following \cite[Sec.2]{ant-bal-semiframes1} one defines the synthesis operator  as
\be
{\sf D}_{\psi}F =  \int_X  F(x) \,\psi_{x} \; \ud\mu(x),  \;\;F\in L^{2}(X, \ud\mu) ,
\label{eq:synthmap2}
\end{equation}
on the domain of all elements $F$ for which the integral in \eqref{eq:synthmap2} converges weakly in $\H$,
and  then $S_\psi:={\sf D}_{\psi} {\sf C}_{\psi}$.
With this definition, it is shown in \cite[Sec.2]{ant-bal-semiframes1}  that $S_\psi$   is unbounded and $S_\psi^{-1}$ is bounded. }

All these objects are studied in detail in our previous papers \cite{ant-bal-semiframes1,ant-bal-semiframes2}. In particular, it is shown
there that a natural notion of duality exists, namely,  two measurable functions $\psi, \phi$ are dual to each other (the relation is
symmetric) if one has
$$
\ip{f}{g} = \int_X  \ip{f}{\psi_{x}} \ip{\phi_{x}} {g}\, \; \ud\mu(x), \; \forall\, f,g \in \H.
$$
\smallskip

\section{Hilbert spaces generated by a reproducing pair}
\label{sec-Banach}
The couple of weakly measurable functions $(\psi, \phi)$   is called a \emph{reproducing pair} if
\\[1mm]
(a) The sesquilinear form
\be\label{eq:form}
\Omega_{\psi, \phi}(f,g) = \int_X \ip{f}{\psi_x} \ip{\phi_x}{g} \ud\mu(x)
\en
 is well-defined and bounded on $\H \times \H$, {that is, $| \Omega_{\psi, \phi}(f,g) | \leq c \norm{}{f}\norm{}{g}$,} {  for some $c>0$.
 \\[1mm]
(b) The corresponding bounded operator $S_{\psi, \phi}$ belongs to $GL(\H)$.
\medskip

Under these hypotheses, one has
\be\label{eq-Sf}
S_{\psi, \phi}f = \int_X \ip{f}{\psi_x} \phi_x  \ud\mu(x) , \; \forall\,f\in\H,
\en
the integral on the r.h.s. being defined in weak sense.

{If $\psi = \phi$, we recover the notion  of continuous frame. }
\medskip

In this section we will study normed spaces constructed from weakly measurable functions and show that for reproducing pairs these spaces enjoy natural duality properties.

\subsection{Construction and characterization of the spaces $V_\phi(X,\mu)$}
\label{subsec-banach}
Let $\phi$ be a weakly measurable function and let us denote by  $\V_\phi(X, \mu) $ the space of all  measurable functions  $\xi  : X \to \CN$ such that the integral
$\int_X  \xi(x)  \ip{\phi_x}{g} \ud\mu(x)$
exists for every $g\in \H$ and defines a bounded conjugate linear functional on $\H$, i.e., $\exists\; c>0$ such that
\be\label{eq-Vphi}
\left|  \int_X  \xi(x)  \ip{\phi_x}{g} \ud\mu(x) \right| \leq c \norm{}{g}, \; \forall\, g \in \H.
\en

\beex
{If the sesquilinear form $\Omega_{\psi,\phi}$ defined in \eqref{eq:form} is bounded, in particular if $(\psi,\phi)$ is a reproducing pair,
}it is clear that all functions $\xi(x)  = \ip{f}{\psi_x}$ 
belong to  $\V_\phi(X, \mu) $ since, by assumption,
$$
\int_X \ip{f}{\psi_x} \ip{\phi_x}{g} \ud\mu(x)
$$
 exists and is bounded.
\enex

For every $\xi \in \V_\phi(X, \mu) $, there exists a unique vector $h_{\phi,\xi}\in \H$ such that
$$
\int_X  \xi(x)  \ip{\phi_x}{g} \ud\mu(x) = \ip{h_{\phi,\xi}}{g}, \quad \forall g \in \H.
$$
Then  we can define a linear map
\be\label{eq:Tphi}
T_\phi : \xi\in \V_\phi(X, \mu) \mapsto T_\phi   \xi \in \H
\en
in the following weak sense
\be\label{eq:Tphi2}
\ip{T_\phi   \xi}{g}   = \ip{h_{\phi,\xi}}{g} =\int_X  \xi(x)  \ip{\phi_x}{g} \ud\mu(x) , \; { \forall} g\in\H.
\en

{The kernel of $T_\phi$ and the notion of degeneracy will be studied in more detail in Section \ref{sec-nondeg}.}
Accordingly, we   define the following vector space
$$
V_\phi(X, \mu)= \V_\phi(X, \mu)/{\Ker}\,T_\phi.
$$

If $\xi\in V_\phi(X, \mu)$, we put, for short, $\cl{\xi}{\phi}= \xi + {\Ker}\,T_\phi$ and define
\be \label{eq-normphi}
\norm{\phi}{\cl{\xi}{\phi}} := \sup_{\norm{}{g}\leq 1 } \left|  \int_X  \xi(x)  \ip{\phi_x}{g} \ud\mu(x)  \right|.
\en

It is easy to see that the left hand side does not depend on the particular representative of $\cl{\xi}{\phi}$.

The following result is immediate.
\begin{prop}Let $\phi$ be a
weakly measurable function.
Then $V_\phi(X, \mu) $ is a normed space with respect to $\norm{\phi}{\cdot}$ and
the map $\T_\phi:V_\phi(X,\mu)\rightarrow \H$, $\T_\phi[\xi]_\phi:= T_\phi \xi$
is a well-defined isometry  of $V_\phi(X, \mu)$ into $\H$. \end{prop}
{
Since $\T_\phi:V_\phi(X,\mu)\rightarrow \H$ is an isometry, we can define on $V_\phi(X,\mu)$ an inner product by setting
$$
\ip{\cl{\xi}{\phi}}{\cl{\eta}{\phi}}_{(\phi)}: =\ip{\T_\phi\cl{\xi}{\phi}}{\T_\phi\cl{\eta}{\phi}}, \; \cl{\xi}{\phi}, \cl{\eta}{\phi}{ \in V_\phi(X,\mu)}.
$$
Using \eqref{eq:Tphi2}, we get, more explicitly
\begin{align*}
\ip{\cl{\xi}{\phi}}{\cl{\eta}{\phi}}_{(\phi)}&=\int_X\xi(x)\left(\int_X\ov{\eta(y)}\ip{\phi_x}{\phi_y}\ud\mu(y) \right)\ud\mu(x)
\\
&=\int_X \ov{\eta(y)}\left(\int_X \xi(x) \ip{\phi_x}{\phi_y}\ud\mu(x) \right)\ud\mu(y)
\end{align*}
It is easy   to see that the norm defined by $ \ip{\cdot}{\cdot}_{(\phi)}$ coincides with the norm $\| \cdot\|_\phi$ defined in \eqref{eq-normphi}.
Thus $V_\phi(X, \mu) $ is a inner product (pre-Hilbert) space.}

Let us denote by   {$V_\phi(X,\mu)^\ast$} the {  Hilbert} dual space  of  $V_\phi(X, \mu) $, that is, the set of continuous linear  functionals on $V_\phi(X, \mu) $. The norm $\|\cdot\|_{\phi^\ast}$ of $V_\phi(X,\mu)^\ast$ is defined, as usual, by
$$\|F\|_{\phi^\ast}=\sup_{\|[\xi]_\phi\|_\phi\leq 1}|F([\xi]_\phi)|.$$
Now we define a linear map $C_\phi: \H \to V_\phi(X,\mu)^\ast$ by
\begin{equation}\label{eq-analysis}(C_\phi f)([\xi]_\phi):= \int_X {\xi(x)}\ip{\phi_x}{f}\ud\mu(x),\end{equation}
which will take the role of the analysis operator ${\sf C}_\phi$ of Section \ref{prel}.

Of course, \eqref{eq-analysis} means that $(C_\phi f)([\xi]_\phi)=\ip{T_\phi \xi}{f}=\ip{\T_\phi [\xi]_\phi}{f}$, for every $f \in \H$.

By \eqref{eq-Vphi} it follows that $C_\phi$ is continuous and, by the definition itself
$C_\phi=\T_\phi ^*$, the adjoint map of $\T_\phi$.
This relation  implies that
\begin{equation}\label{eq-direct} \H= \overline{\ran \T_\phi} \oplus \Ker C_\phi, \end{equation}
and also that $C_\phi^*=\T_\phi ^{**}=\T_\phi$, if   $V_\phi(X,\mu)$  is {complete}.

By modifying in an obvious way the definition given in Section \ref{prel}, we say that $\phi$ is {\em $\mu$-total} if $\Ker C_\phi=\{0\}$.

\berem Whenever no confusion may arise, we will omit the explicit indication of residues classes and write simply, for instance,
 $\xi \in V_\phi(X, \mu)$ instead of $[\xi]_\phi \in V_\phi(X, \mu)$. Similarly, for the operator $C_\phi$ introduced in \eqref{eq-analysis}, we will often identify $C_\phi f$, $f \in \H$, with $\ip{f}{\phi_x}$, as a shortcut to   $(C_\phi f)(\xi)= \int_X \xi (x) \ip{\phi_x}{f}\ud\mu(x)$.

\enrem

{ \begin{prop}\label{theo_new1}
The following statements are equivalent.
\begin{itemize}
\item[(i)] $V_\phi(X,\mu)[\ip{\cdot}{\cdot}_{(\phi)}]$ is a Hilbert space.
\item[(ii)] $\T_\phi$ has closed range.
\end{itemize}
\end{prop}

\begin{proof} (i)$\Rightarrow$(ii): Since $V_\phi(X, \mu)$ is complete and $\T_\phi$ is an isometry, $\ran \T_\phi$ is also complete.

\noi (ii)$\Rightarrow$(i): 
{Let $\T_\phi $ have closed range. Then $\T_\phi: V_\phi(X,\mu) \to \ran \T_\phi$ is  isometric} with isometric inverse. Hence, $V_\phi(X, \mu) = \T_\phi ^{-1}(\ran \T_\phi )$ is the isometric image of a complete space,
and therefore it is complete.
\end{proof}
}

\medskip

As a consequence of \eqref{eq-direct} we get
\becor \label{cor34}The following statements hold.
\begin{itemize}
\item[(i)] A weakly measurable function $\phi$ is  $\mu$-total if and only if
$\ran  \widehat T_\phi$ is dense in $\H$.

\item[(ii)] If $V_\phi(X,\mu)$ is a Hilbert space, $\ran \widehat T_\phi$
   is equal to $\H$ if and only if $\phi$ is $\mu$-total.
   \end{itemize}
\encor

\belem If $(\psi,\phi)$ is a  reproducing pair, then $\ran \T_\phi = \H$.
\enlem
\begin{proof}
Since $S_{\psi, \phi}\in GL(\H)$, for every $h\in\H$, there exists a unique $f\in\H$ such that $S_{\psi, \phi}f = h$.
But, by \eqref{eq-Sf}, we get
$$
h =   \int_X \ip{f}{\psi_x} \phi_x  \ud\mu(x) ,
$$
so that
$$
\ip{h}{g} =  \int_X \ip{f}{\psi_x} \ip{\phi_x}{g}  \ud\mu(x) ,  \; \forall\,f,g\in\H,
 $$
 that is, $h = \T_\phi [C_\psi f]_\phi $.
\end{proof}

\medskip

\noi    Notice that, if $(\psi, \phi)$ is a reproducing pair, both functions are necessarily $\mu$-total.

\medskip

Let $(\psi,\phi)$ be a  reproducing pair. Then, corresponding to $\T_\phi$, we introduce the   operator $\widehat{C}_{\psi ,\phi}  :  \H \to V_\phi(X, \mu) $ by
 $\widehat{C}_{\psi ,\phi} f:= [C_\psi f]_\phi $.
We note that the construction can distinguish the equivalence classes generated by the analysis operator. Indeed, we have
$\widehat{C}_{\psi ,\phi} f = \widehat{C}_{\psi ,\phi} f'$
if and only if $f=f'$. To see this, let $\widehat{C}_{\psi ,\phi} f = \widehat{C}_{\psi ,\phi} f'$.
Then
\begin{equation}\label{repres-functional}
0=\int_X \ip{f-f'}{\psi_x}\ip {\phi_x}{g} \ud\mu(x)=\ip{ S_{\psi,\phi}(f-f')}{g},\ \forall g\in\H .
\end{equation}
Since $S_{\psi,\phi}\in GL(\H)$, it follows that $f=f'$.
\smallskip

\subsection{Duality properties of the spaces  $V_\phi(X, \mu) $} \label{subsec-dualbanach} { The space $V_\phi(X, \mu) $ is a Hilbert space, thus it is certainly isomorphic to its dual, {via the Riesz operator}. Nevertheless if $(\psi,\phi)$ is a reproducing pair, the dual of
$V_\phi(X, \mu) $ can be identified with $V_\psi(X, \mu) $ as we shall prove below.} {We emphasize that the duality is taken   with respect to the sesquilinear form
 
\be \label{eq_sesq}
\ipp{\xi}{\eta}:= \int_X \xi(x)  \ov{\eta(x) } \ud\mu(x),
\en  
which coincides with the inner product of $L^2(X, \mu)$ whenever the latter makes sense.}

\betheo\label{theo23}
Let $\phi$ be a  weakly measurable function.
If $F$ is a continuous linear functional on $V_\phi(X, \mu) $, then there exists a unique {$g\in\overline{\M}_\phi$}, the closure of the range of  $\T_\phi$,  such that
\begin{equation}\label{repres-functional2}
F(\cl{\xi}{\phi}) =  \int_X  \xi(x)  \ip{\phi_x}{g} \ud\mu(x) , \; \forall\, \xi \in \V_\phi(X, \mu)
\end{equation}
and   $\norm{\phi^\ast}{F}   = \norm{}{g}$, where $\norm{\phi^\ast}{\cdot}$ denotes the (dual) norm on
$V_\phi(X,\mu)^\ast$.
Moreover, every $g\in \H$ defines a bounded functional  $F$ on $V_\phi(X,\mu)$ with {$\norm{\phi^\ast}{F}  \leq \norm{}{g}$},  by \eqref{repres-functional2}.
In particular, if $g\in \ran \T_\phi$, then $\norm{\phi^\ast}{F}  = \norm{}{g}$.
\entheo
\begin{proof}
Let $F\in V_\phi(X, \mu)^\ast$. Then, there exists $c>0$ such that
$$
| F(\cl{\xi}{\phi}) | \leq c\norm{\phi}{[\xi]_\phi}=c \norm{}{T_\phi \xi}, \; \forall \, \xi \in \V_\phi(X, \mu).
$$
Let $\M_\phi:= \{ T_\phi \xi : \xi \in \V_\phi(X, \mu) \} = \ran  \T_\phi $. Then  $\M_\phi$ is a vector subspace of $\H$, with closure {$\overline{\M}_\phi$}.

Let $\widetilde F$ be the functional defined on $\M_\phi$ by
$$
{\widetilde F}(T_\phi \xi) := F(\cl{\xi}{\phi}), \; \xi \in \V_\phi(X, \mu).
$$
We  notice that $\widetilde F$ is well-defined. Indeed, if $T_\phi \xi = T_\phi \xi'$, then $\xi- \xi'\in {\Ker\,}T_\phi$. Hence, $\cl{\xi}{\phi}=\cl{\xi'}{\phi}$
and $F(\cl{\xi}{\phi})=F(\cl{\xi'}{\phi})$

Hence, $\widetilde F$ is a bounded linear functional   on $\M_\phi$. Thus there exists a unique   {$g\in\overline{\M}_\phi$}  such that
$$
{\widetilde F}(T_\phi \xi)  = \ip{\T_\phi \cl{\xi}{\phi} }{g} = \int_X  \xi(x)  \ip{\phi_x}{g} \ud\mu(x)
$$
and $\|g\|=\|{\widetilde F}\|$.

In conclusion,
$$
F(\cl{\xi}{\phi}) =  \int_X  \xi(x)  \ip{\phi_x}{g} \ud\mu(x),\; \forall\, \xi \in \V_\phi(X, \mu) .
$$
and {$\norm{\phi^\ast}{F}  = \norm{}{g}$}.

Moreover, every $g\in \H$ obviously defines a bounded linear functional $F$ by \eqref{repres-functional2} as  $|F([\xi]_\phi)|\leq \norm{}{g}\norm{\phi}{[\xi]_\phi}$. 
This inequality implies that $\norm{\phi^\ast}{F}\leq \norm{}{g}$. In particular, if $g \in \ran \T_\phi$, then there exists $[\xi]_\phi\in \V_\phi(X, \mu)$, $\|[\xi]_\phi\|_\phi=1$, such that $\T_\phi [\xi]_\phi= g \|g\|^{-1}$. Hence $F([\xi]_\phi)=\ip{\T_\phi [\xi]_\phi}{g}=\|g\|$. This concludes the proof.
\end{proof}

\medskip

\becor\label{isomcor}
Let $\phi$ be a $\mu$-total weakly measurable function, then $C_\phi:\H\rightarrow V_\phi(X,\mu)^\ast$ is an isometric isomorphism.
\encor
\begin{proof}
$C_\phi$ is surjective  by Theorem \ref{theo23}. As $\phi$ is  $\mu$-total, it follows by Corollary \ref{cor34}
 that $\ran \widehat T_\phi$ is dense in $\H$. Consequently, for $f\in\H$ it follows that
 $$
 \norm{\phi^\ast}{C_\phi f}=\sup_{\norm{\phi}{[\xi]_\phi}=1}\left|\int_X \xi(x)\ip{\phi_x}{f}\ud\mu(x)\right|
 =\sup_{\norm{\phi}{[\xi]_\phi}=1}|\ip{\widehat T_\phi \xi}{f}|=\sup_{\norm{}{g}=1,\ g\in\ran \widehat T_\phi}|\ip{g}{f}|=\norm{}{f}.
 $$
\end{proof}

 { \berem\label{rem217}
It turns out that $C_\phi$ being an isometric isomorphism is not sufficient to guarantee that $V_\psi(X,\mu)$ is complete.
We will  see a counterexample in Sec. \ref{subsub-upperframe}.
\enrem
 }

\betheo \label{representation-of-F}
 If  $(\psi,\phi)$ is a reproducing pair,   then every bounded linear functional $F$ on $V_\phi(X, \mu)$,  {i.e.,  $F\in V_\phi(X, \mu)^\ast$,} can be represented as
\be\label{eq-dual}
F(\cl{\xi}{\phi}) = \int_X \xi(x)  \ov{\eta(x) } \ud\mu(x), \; \forall\,\cl{\xi}{\phi}\in V_\phi(X, \mu),
\en
with $\eta\in \V_\psi(X, \mu)$. The  {residue} class $\cl{\eta}{\psi}\in V_\psi(X, \mu)$ is uniquely determined.
\entheo
\begin{proof}
By Theorem \ref{theo23}, we have the representation
$$
F(\xi) =  \int_X  \xi(x)  \ip{\phi_x}{g} \ud\mu(x)  .
$$
It is easily seen that $\eta(x) = \ip{g}{\phi_x} \in \V_\psi(X, \mu)$.

It remains to prove uniqueness. Suppose that
$$
F(\xi) = \int_X \xi(x)  \ov{\eta'(x) } \ud\mu(x) .
$$
Then
$$
\int_X \xi(x)  (\ov{\eta'(x)}- \ov{\eta(x)}) \ud\mu(x) =0.
$$
Now the function $\xi(x) $ is arbitrary. Hence, taking in particular for $\xi(x) $ the functions  $\ip{f}{\psi_x}\in \V(X,\mu),$ $f\in\H$,
we get $\cl{\eta}{\psi}=\cl{\eta'}{\psi}$.
\end{proof}
\medskip

The lesson of the previous statements is that the map
\be\label{map-j}
  j : F\in V_\phi(X, \mu)^\ast \mapsto \cl{\eta}{\psi} \in V_\psi(X, \mu)
\en
is well-defined and conjugate linear. On the other hand, $j(F) = j(F')$ implies easily $F=F'$.
Therefore $V_\phi(X, \mu)^\ast$ can be identified with a   closed subspace  of $\ov{V_\psi}(X, \mu):=\{\ov{\cl{\xi}{\psi}} : \xi\in \V_\psi(X, \mu)\}$, the conjugate
space of $V_\psi(X, \mu)$.

Now we want to prove that the spaces $V_\phi(X, \mu)^\ast$ and $\ov{V_\psi}(X, \mu)$ can be identified.
 {To that effect,  we will first prove two
auxiliary lemmas.
\belem\label{first-auxil-lemma}
Let $(\psi,\phi)$ be a reproducing pair.  Then $\ran \widehat C_{\psi,\phi}$ is closed in $V_\phi(X,\mu)[\norm{\phi}{\cdot}]$.
In particular, there exist ${\sf m},{\sf M}>0$ such that
\be\label{eq-triple}
 {\sf m} \norm{}{f}  \leq \norm{\phi}{\widehat C_{\psi,\phi}f}\leq {\sf M}\norm{}{f} , \; \forall\,f\in\H.
\en
Moreover, every $\cl{\eta}{\psi} \in V_\psi(X, \mu)$ defines a  {bounded   linear functional}
on the closed subspace $\ran \widehat C_{\psi,\phi}[\norm{\phi}{\cdot}]$.
\enlem
\begin{proof}
Since $S_{\psi, \phi}\in GL(\H)$, we have, for $f\in\H$,
\begin{align*}
\norm{}{f} &\leq  \norm{}{S_{\psi, \phi}^{-1}} \, \norm{}{S_{\psi, \phi}f} =\norm{}{S_{\psi, \phi}^{-1}} \, \sup_{\norm{}{g}\leq 1}|\ip{S_{\psi, \phi}f}{g}|
\\
& =\norm{}{S_{\psi, \phi}^{-1}} \sup_{\norm{}{g}\leq 1} \left|\int_X \ip{f}{\psi_x}\ip{\phi_x} {g} \ud\mu(x)\right| = \norm{}{S_{\psi, \phi}^{-1}} \norm{\phi}{\cl{\ip{f}{\psi(\cdot)}}{\phi}} = \norm{}{S_{\psi, \phi}^{-1}} \norm{\phi}{\widehat C_{\psi,\phi}f}.
\end{align*}
This relation implies that $\ran \widehat C_{\psi,\phi}$ is closed in $V_\phi(X,\mu)[\|\cdot\|_\phi]$.
On the other hand,
\begin{align*}
\norm{\phi}{\widehat C_{\psi,\phi}f} & =\sup_{\norm{}{g}\leq1}\left|\int_X\ip {f}{\psi_x} \ip{\phi_x}{g} \ud\mu(x)\right|
\\ &= \sup_{\norm{}{g}\leq1}| \ip{S_{\psi,\phi} f}{g}|  = \norm{}{S_{\psi,\phi} f} \leq \norm{}{f} \norm{}{S_{\psi,\phi}} .
\end{align*}
Next, let $\eta \in V_\psi(X, \mu)$. Then, by definition, $\int_X \ip{f}{\psi_x} \ov{\eta(x)} \ud\mu(x) $ exists and
defines a bounded linear functional on $\H$, i.e.,
$$
\left| \int_X \ip{f}{\psi_x} \ov{\eta(x)} \ud\mu(x) \right| \leq c \norm{}{f} , \; \forall\,f\in\H.
$$
By the definition of $\norm{\psi}{\cdot}$, we have, more precisely,
$$
\left| \int_X \ip{f}{\psi_x} \ov{\eta(x)} \ud\mu(x) \right|  \leq \norm{}{f} \norm{\psi}{\eta}, \; \forall\,f\in\H.
$$
Hence,
$$
\left| \int_X    \ip{f}{\psi_x} \ov{\eta(x)}
 \ud\mu(x) \right|  \leq \norm{}{S_{\psi, \phi}^{-1}} \norm{\phi}{ \widehat C_{\psi,\phi}f}\norm{\psi}{\eta}, \; \forall\,f\in\H, \eta\in \V_\psi(X, \mu).
$$
Thus, by \eqref{eq-dual},
$\cl{\eta}{\psi}$ defines a  {bounded   linear functional} on the space  $ \ran \widehat C_{\psi,\phi}=\ran C_\psi/{\Ker} \,T_\phi $.
\makebox[1cm]{}\end{proof}}
 \medskip

 If $(\psi,\phi)$ {  is} a reproducing pair and $\norm{\phi}{\widehat C_{\psi,\phi} f} = \norm{}{f}$, then $S_{\psi, \phi}$ is an isometry, since one has, for every $f\in\H$,
 $$
  \norm{}{f} = \sup_{ \norm{}{g}=1} \left| \int_X  \ip{f}{\psi (x)}  \ip{\phi_x}{g} \ud\mu(x)  \right|
  = \sup_{ \norm{}{g}=1} {|\ip{S_{\psi, \phi}f}{g} |} = \norm{}{S_{\psi, \phi} f}.
 $$

 {\belem\label{lem-dense}
 Let $(\psi,\phi)$ be a reproducing pair. Then ${\ran}\widehat{C}_{\psi ,\phi} $ is dense in $V_\phi (X, \mu)$.
\end{lemma}
\begin{proof}Were it not so, there would be a nonzero $F\in V_\phi (X, \mu)^*$ such that $F(\ip{f}{\psi (\cdot)})=0$ for every
$f\in \H$.
By Theorem \ref{theo23}, there exists $g\in  {\H\backslash\{0\}}$,  such that
$$
F(\xi) =  \int_X  \xi(x)  \ip{\phi_x}{g} \ud\mu(x) , \; \forall\, \xi \in V_\phi(X, \mu) .
$$
Then,
$$
F(\ip{f}{\psi (\cdot)}) =  \int_X  \ip{f}{\psi (x)}  \ip{\phi_x}{g} \ud\mu(x) =0 , \; \forall\, f \in \H .
$$
This implies that $\ip{S_{\psi,\phi}f}{g}=0$, for every $f\in \H$.   This in turn implies that $g=0$, which is a contradiction.
\end{proof}
 \betheo \label{theorem-dual-space}
 If  $(\psi,\phi)$ is a reproducing pair,   the map $j$ defined in \eqref{map-j} is surjective. Hence
$V_\phi(X, \mu)^\ast \simeq \ov{V_\psi}(X, \mu)$,
 where $\simeq$ denotes a   bounded    isomorphism and the norm
$\norm{\psi}{\cdot}$ is the dual norm of $\norm{\phi}{\cdot}$.
Moreover, $\ran \widehat C_{\psi,\phi}[\norm{\phi}{\cdot}]=V_\phi(X,\mu)[\|\cdot\|_\phi]$
and $\ran \widehat C_{\phi,\psi}[\norm{\psi}{\cdot}]=V_\psi(X,\mu)[\|\cdot\|_\psi]$.
\entheo
\begin{proof}
By Lemma \ref{first-auxil-lemma}, $\ran \widehat C_{\psi,\phi}$ is closed in $V_\phi(X,\mu)[\|\cdot\|_\phi]$.
By Lemma \ref{lem-dense}, it is also dense. Hence, $\ran \widehat C_{\psi,\phi}[\norm{\phi}{\cdot}]$ and  {${V_\phi}(X,\mu)[\|\cdot\|_\phi]$ }coincide.
Now, the map $j$ is surjective as every $\eta\in V_\psi$ defines a {bounded  linear functional} on
$V_\phi(X,\mu)[\|\cdot\|_\phi]$. \makebox[1cm]{}
\end{proof}
\medskip

 By Theorems \ref{representation-of-F}   and \ref{theorem-dual-space}, it follows that, if $(\psi,\phi)$ is a reproducing pair,
then for every $\eta \in V_\psi(X,\mu)$, there exists $g \in \H$ such that $\eta= \ip{\phi(\cdot)}{g}$.

In conclusion, we may state
{
\betheo\label{theo-dual}
If $(\psi,\phi)$ is a reproducing pair,  the spaces $V_\phi(X, \mu)$ and $V_\psi(X, \mu)$ are both Hilbert spaces, {conjugate dual} of each other with respect to the sesquilinear form \eqref{eq_sesq}.
\entheo
}
\becor\label{cor-Hil}
If {$(\psi,\phi)$ is a reproducing pair and} $\phi = \psi$, then $\psi$ is a continuous frame and $V_\psi(X, \mu)$ is a closed subspace of $L^2(X, \mu)$.
\encor
\begin{proof}
 Since the duality takes place with respect to the $L^2$ inner product, $V_\psi(X, \mu)$ is a subspace of $L^2(X, \mu)$. The equality $\ran \widehat C_{\psi, {\psi}}=V_\psi(X, \mu)$ and the fact that $\widehat C_{\psi, {\psi}}$ is bounded from below with respect to the $L^2$-norm imply that it is closed.

\end{proof}

Actually Theorem \ref{theo-dual} has an inverse. Indeed:
\betheo
\label{theo-111}
Let $\phi$ and $\psi$ be weakly measurable and $\mu$-total. Then,  the couple $(\psi,\phi)$ is a reproducing pair if and only if
$V_\phi(X, \mu)$ and $V_\psi(X, \mu)$ are Hilbert spaces,  {conjugate dual}  of each other with respect to the sesquilinear form \eqref{eq_sesq}.
\entheo
\begin{proof}
The `if' part is Theorem \ref{theo-dual}. Let now $V_\phi(X, \mu)$ and $V_\psi(X, \mu)$ be Hilbert spaces in {conjugate duality.} Consider the sesquilinear form
$$
\Omega_{\psi,\phi}(f,g) = \int_X \ip{f}{\psi_x} \ip{\phi_x}{g} \ud\mu(x), \; f,g\in \H.
$$
By the definition of the norms $\norm{\phi}{\cdot},\norm{\psi}{\cdot}$ and the duality condition, we have, for every $f,g\in\H$,  the two inequalities
\begin{align*}
|\Omega_{\psi,\phi}(f,g)| &\leq \norm{\phi}{\cl{\ip{f}{\psi(\cdot)}}{\phi}}\norm{}{g}, \\
|\Omega_{\psi,\phi}(f,g)| &\leq \norm{\psi}{\cl{\ip{g}{\phi(\cdot)}}{\psi}}\norm{}{f}.
\end{align*}
This means the form $\Omega_{\psi,\phi}$ is separately continuous, hence continuous. Therefore there exists a bounded operator $S_{\psi,\phi}$ such that
$\Omega_{\psi,\phi}(f,g) = \ip{S_{\psi,\phi} f}{g}.$ First the operator $S_{\psi,\phi}$ is injective. Indeed,  since
$C_\phi^\ast=\widehat T_\phi$, we have
$$
\ip{S_{\psi,\phi} f}{g} = \ip{C_\psi f}{C_\phi g} =  \ip{ \widehat{C}_{\psi ,\phi}  f}{C_\phi g} = \ip{\T_\phi \widehat{C}_{\psi ,\phi}  f}{ g},\; \forall f,g\in \H.
$$
Now $\T_\phi$ is isometric and $\widehat{C}_{\psi ,\phi} $ is injective, hence $\T_\phi \widehat{C}_{\psi ,\phi} f =0$ implies $f=0$.
 Next,  $S_{\psi,\phi}$ is also surjective, by Corollary \ref{cor34}.
Hence $S_{\psi,\phi}$ belongs to $GL(\H).$
\end{proof}
\medskip

{ \berem If the couple $(\psi,\phi)$ is a reproducing pair, then $V_\phi(X, \mu)$ and $V_\psi(X, \mu)$ are Hilbert spaces,  {conjugate dual}  of each other with respect to $\ipp{\cdot}{\cdot}$. Thus, every $[\eta]_\psi \in  V_\psi(X, \mu)$ determines a linear functional $F_\eta$ on $V_\phi(X, \mu)$ by
$$F_\eta([\xi]_\phi)= \int_X \xi(x) \ov{\eta(x)}\ud\mu(x)=\ipp{\xi}{\eta}.$$
On the other hand (Riesz's lemma) there exists a unique $[\eta']_\phi\in V_\phi(X, \mu)$ such that
$$ F_\eta([\xi]_\phi)= \ip{\cl{\xi}{\phi}}{\cl{\eta'}{\phi}}_{(\phi)}=\int_X\xi(x)\left(\int_X\ov{\eta'(y)}\ip{\phi_x}{\phi_y}\ud\mu(y) \right)\ud\mu(x).$$
Define $N:[\eta]_\psi \in  V_\psi(X, \mu)\to [\eta']_\phi\in V_\phi(X, \mu)$. Then,
$$\ipp{\xi}{\eta}=\ip{[\xi]_\phi}{N[\eta]_\psi}_{(\phi)}, \quad \forall [\xi]_\phi \in V_\phi(X, \mu), [\eta]_\psi \in  V_\psi(X, \mu). $$

In the very same way we can define an operator $M:V_\phi(X, \mu) \to V_\psi(X, \mu)$ such that
$$\ipp{\xi}{\eta}=\ip{M[\xi]_\phi}{[\eta]_\psi}_{(\psi)}, \quad \forall [\xi]_\phi \in V_\phi(X, \mu), [\eta]_\psi \in  V_\psi(X, \mu). $$
Then it is clear that $N^\ast=M$. Moreover, $N$ is isometric. Hence, $N^\ast=N^{-1}=M.$ From the above equalities we get an explicit form for $N^{-1}$
$$(N^{-1} \cl{\eta'}{\phi})(x)= \int_X{\eta'(y)}\ip{\phi_y}{\phi_x}\ud\mu(y).$$
\enrem
}

In addition to Lemma \ref{theorem-dual-space}, there is another characterization of the space $V_\psi(X, \mu)$, in terms of an eigenvalue equation, based on the fact that
$\ip{S_{\psi,\phi}^{-1}\phi_y}{\psi_x}$ is a reproducing kernel \cite[Prop.3]{speck-bal}.
\beprop
Let $(\psi,\phi)$  be a reproducing pair. Let $\xi\in\mathcal{V}_\psi(X,\mu)$ and consider the eigenvalue equation
\be\label{reprodker}
\int_X \xi(y)\ip{S_{\psi,\phi}^{-1}\phi_y}{\psi_x}  \ud\mu(y)=\lambda \xi(x).
\en
Then  $ \xi\in \ran C_\phi\ \Leftrightarrow\ \lambda =1$ and $\xi\in \Ker T_\psi\ \Leftrightarrow\ \lambda=0$.
Moreover, there are no other eigenvalues.
\enprop

 \section{Existence of reproducing partners}
 \label{sec-partners}

 Next we present a criterion towards the existence of a specific {dual} partner to a given measurable function.
We remind that the basic sesquilinear form $\ipp{\cdot}{\cdot}$ is given by \eqref{eq_sesq}.

 \betheo\label{theo-partner}
Let $\phi$ be a weakly measurable function and $e=\{e_n\}_{n\in\nN}$ an orthonormal basis of $\H$. There exists another measurable
function $\psi$, such that $(\psi,\phi)$ is a reproducing pair if and only if $\ran \widehat T_\phi =\H$
and there exists a family $\{\xi_n\}_{n\in\nN}\subset \V_\phi(X,\mu)$ such that
\begin{equation}\label{second-assumption}
[\xi_n]_\phi=[\widehat T_\phi^{-1} e_n]_\phi,\ \forall n\in\nN,\hspace{0.5cm} \text{and} \hspace{0.5cm}
\sum_{n\in\nN}|\xi_n(x)|^2<\infty,\ \text{for a.e.}\ x\in X.
\end{equation}
\entheo

\begin{proof}
If $\ran \widehat T_\phi =\H$, then $V_\phi(X,\mu)$ is a Hilbert space, $C_\phi:\H\to V_\phi^\ast(X,\mu)$ is  an isometric
isomorphism and $C_\phi^\ast=\T_\phi$.
Hence, for $f,g\in\H$,  one has
\begin{equation}\label{eqn3.3}
\begin{aligned}
\ip{f}{g}&=\ip{C_\phi^{-1} C_\phi f}{g}=\ipp{C_\phi f}{(C_\phi^{-1})^\ast g}=
\ipp{C_\phi f}{(C_\phi^{-1})^\ast\Big(\sum_{n\in\nN} \ip{g}{e_n}e_n\Big)}\\
&=\ipp{C_\phi f}{\sum_{n\in\nN}\ip{g}{e_n}(C_\phi^{-1})^\ast e_n}=\ipp{C_\phi f}{\sum_{n\in\nN}\ip{g}{e_n}\widehat T_\phi^{-1}e_n}
\end{aligned}
\end{equation}
where $\{e_n\}_{n\in\nN}$ is an orthonormal basis of $\H$.

Let $(\psi,\phi)$ be a reproducing pair. As $S_{\psi,\phi}\in GL(\H)$, it immediately follows that $\ran\T_\phi=\H$
and thus \eqref{eqn3.3} holds.
 For the sake of simplicity assume that $S_{\psi,\phi}=I$.
Using \eqref{eqn3.3} we get
$$
\ipp{C_\phi f}{C_\psi g}=\ipp{C_\phi f}{\sum_{n\in\nN}\widehat T_\phi^{-1}e_n\ip{g}{e_n}}, \; \forall \,f,g\in\H,
$$
and, consequently,
$$
C_\psi g=\sum_{n\in\nN}\ip{g}{e_n}C_\psi e_n=\sum_{n\in\nN}\ip{g}{e_n}\widehat T_\phi^{-1}e_n,\;\forall \,g\in\H.
$$
In particular,  the choice $g=e_n$ implies $[C_\psi e_n]_\phi=[\widehat T_\phi^{-1}e_n]_\phi,\ \forall n\in\nN$. Moreover,
$C_\psi e(x):=\{C_\psi e_n(x)\}_{n\in\nN}\in \ell^2(\nN)$ for almost every $x\in X$, since $\norm{\ell^2}{C_\psi e(x)}=\norm{}{\psi_x}$.

Conversely, if $\ran\T_\phi=\H$  the following holds weakly by \eqref{eqn3.3}
$$
f=\int_X C_\phi f(x)\Big(\sum_{n\in\nN}\overline{\widehat T_\phi^{-1}e_n(x)}e_n\Big) \ud\mu(x),\;\forall \,f\in\H.
$$
By \eqref{second-assumption} we can find $\{\xi_n\}_{n\in\nN}\subset \V_\phi(X,\mu)$ such that
$$
f=\int_X C_\phi f(x)\Big(\sum_{n\in\nN}\overline{\xi_n(x)}e_n\Big) \ud\mu(x),\;\forall \,f\in\H,
$$
holds weakly and $\psi_x:=\sum_{n\in\nN}{\overline{\xi_n(x)}}e_n$ is a well defined vector in
$\H$ for almost every $x\in X$.
\end{proof}

 {\berem
If $\phi$ is in fact a frame, then the reproducing partner $\psi$ given by the proof of Theorem 3.1 is also a frame.
To see this, we first observe that if $\psi$ is an upper semi-frame (Bessel mapping), then its reproducing partner $\phi$ is necessarily
a lower semi-frame   {\cite[Lemma 2.5]{ant-bal-semiframes1}}. The operator $\widehat T_\phi^{-1}$ is given by
$C_\phi S_\phi^{-1}$.
Hence, for some $\gamma>0$ and for every $f\in\H$,
\begin{align*}
\norm{2}{C_\psi f}^2&=\int_X\Big|\sum_{n\in\nN}\ov{\ip{{S^{-1}_\phi e_n}}{\phi_x}}\ip{f}{e_n}\Big|^2 \ud\mu(x)\\ &=
\int_X\left|\sum_{n\in\nN}\ip{f}{e_n}\ip{e_n}{(S^{-1})^\ast\phi_x}\right|^2 \ud\mu(x) 
\leq \gamma \norm{}{S_\phi^{-1}}^2\|f\|^2.
\end{align*}
Observe that there may exist a reproducing partner $\psi$ which is not Bessel.
\enrem}

Given the weakly measurable function $\phi$, the fact that  $(\psi,\phi)$ is a reproducing pair does not determine the function $\psi$ uniquely. Indeed we have :
\begin{theorem}
Let $(\psi,\phi)$ be a reproducing pair, then $(\theta,\phi)$ is a reproducing pair if and only if $\theta=A\psi+\theta_0$, where $A\in GL(\H)$ and
$[\ip{f}{\theta_0(\cdot)}]_\phi=[0]_\phi,\ \forall f\in\H , i.e., \widehat C_{\theta_0,\phi}f = 0, \forall\, f\in\H$.
\end{theorem}
\begin{proof}
If $\theta=A\psi+\theta_0$ as above, then $S_{\theta,\phi}f =\widehat T_\phi (\widehat C_{A\psi,\phi}+\widehat C_{\theta_0,\phi})f=
\widehat T_\phi (\widehat C_{A\psi,\phi}f) =  \widehat T_\phi \widehat C_{\psi,\phi}  A^\ast f =S_{\psi,\phi}A^\ast f$, hence
$S_{\theta,\phi}= S_{\psi,\phi}A^\ast \in GL(\H)$.

Conversely, assume that $(\theta,\phi)$ is a reproducing pair. By Theorem \ref{theorem-dual-space},
we have  {$V_\phi(X,\mu)=\ran C_\psi/\Ker T_\phi=\ran C_\theta/\Ker T_\phi$}, i.e., for every $f\in \H$ there exists $g\in \H$ such that $[C_\theta f]_\phi=[C_\psi g]_\phi$.
Then, using successively the definition of $S_{\phi,\theta}$, the relation $[C_\theta f]_\phi=[C_\psi g]_\phi$ and the reproducing kernel \eqref{reprodker}, we obtain
\begin{align*}
&\ip{ f}{S_{\phi,\theta}(S_{\psi,\phi}^{-1})^\ast\psi(\cdot)}
=\int_X \ip{f}{\theta(x)} \ip{\phi_x}{(S_{\psi,\phi}^{-1})^\ast\psi(\cdot)}\ud\mu(x) \;
\\
&\qquad=\int_X \ip{g}{\psi (x)} \ip {\phi_x}{(S_{\psi,\phi}^{-1})^\ast\psi(\cdot)} \ud\mu(x)
= \ip{ g}{\psi(\cdot)} = \ip{ f}{\theta(\cdot)} \, , \;\forall\, f\in\H.
\end{align*}
This means that, for all $f\in\H$, we have $[C_\theta f]_\phi=[C_{A\psi} f]_\phi$ or, equivalently, $ \widehat C_{\theta,\phi}= \widehat C_{A\psi,\phi}$,
where $A:= S_{\phi,\theta}(S_{\psi,\phi}^{-1})^\ast \in GL(\H)$. Moreover, $C_\theta f(x)=C_{A\psi} f(x)+F(f,x)$ for a.e. $x\in X$ and every $f\in\H$, where
$F(f,\cdot)\in \Ker T_\phi$, i.e., $F(f,x)=\ip{f}{(\theta-A\psi)(x)}=:\ip{ f}{\theta_0(x)}$.
\end{proof}

\section{Nondegenerate systems}
\label{sec-nondeg}

The measurable function $\phi$ is said to be  \emph{$\mu$-independent} if  $\Ker T_\phi = \{0\}$, that is, if
it satisfies the following condition
\be\label{eq(k)}
  \int_X \xi(x) \ip{\phi_x}{g} \ud\mu(x) = 0,\ \forall\, g\in\H , \mbox{ implies }  \xi(x) = 0  \mbox{ a.e.}.
\en
In that case, of course, ${\V}_\phi(X,\mu)=V_\phi(X,\mu)$. This definition is modeled on that of $\omega$-independence of sequences, introduced in \cite[Def.3.1.2]{christ}.
The function  $\phi$ is called  \emph{$\mu$-nondegenerate}
if it is both $\mu$-total and $\mu$-independent.

\beprop\label{bessel-propos}
 {Let $(\psi,\phi)$ be a reproducing pair, where {$\phi$ is Bessel}, and assume $(\ran C_\psi \cap L^2(X,\ud\mu))^\bot \neq \{0\}$.}
Then $\phi$ is not $\mu$-independent, hence  it is $\mu$-degenerate.
\enprop
\begin{proof}
Let us assume that $\phi$ is $\mu$-independent and,  without loss of generality,  that $S_{\psi,\phi}=I$ (that is, $\phi$ and $\psi$ are
dual of each other).
Take $F\in (\ran C_\psi {\cap L^2(X,\ud\mu)})^\bot\backslash\{0\}$. As $\phi$ is $\mu$-independent, it follows that  ${\sf D}_\phi F\neq 0$ and
consequently $F'=C_\psi {\sf D}_\phi F\neq 0$ since $\psi$ is $\mu$-total. Moreover, $F-F'\neq 0$
since $F\in (\ran C_\psi\ {\cap L^2(X,\ud\mu)})^\bot$ and $F'\in C_\psi(\H)$. Hence we get
$$
\int_X (F(x)-F'(x))\ip{ \phi_x}{g}\ud\mu(x)= \ip{ {\sf D}_\phi F- {T_\phi C_\psi}  {\sf D}_\phi F}{g}=0,\ \forall \,g\in\H,
$$
since $T_\phi C_\psi =S_{\psi,\phi}=I$, and this   contradicts the assumption of $\mu$-independence of $\phi$.
\end{proof}
\medskip

 Actually there is more. Assume  that   $\psi$ is an upper semi-frame (i.e., a Bessel map). Then $\phi$ is a lower semi-frame
\cite[Lemma 2.5]{ant-bal-semiframes1} (they can both be frames). Then, if $(X,\mu)$ is a nonatomic measure space, it follows from
 \cite[Theorem 2]{hosseini}  that $\mathrm{dim} (\ran C_\phi {\cap L^2(X,\ud\mu)})^\bot=\infty$.

\medskip

Intuitively, $\mu$-nondegeneracy  occurs only for discrete
 systems (atomic measure) or continuous systems closely related to discrete ones, called continuous orthonormal bases in \cite{arefija}
 and studied in \cite{askari,gab-han}. Incidentally, in the discrete case, similar considerations have been  extended to
 rigged \hs s     in recent papers by Bellomonte and one of us \cite{bello, bello-trap}.

\section{Examples}
\label{sec-examples}

In this section, we present a few concrete examples of the construction of Section \ref{sec-Banach}.
We begin with discrete examples, that is, $X=\NN$ with the counting measure.

\subsection{Discrete examples}

\subsubsection{Orthonormal basis}
Let $e=\{e_n\}_{n\in\NN}$ be an orthonormal basis, then $\V_e(\NN)=V_e(\NN)=\ell^2(\NN)$. Indeed, for $\xi\in \V_e(\NN)$,
we have
$$
\Big|\sum_{n\in\NN}\xi_n\ip{e_n}{g}\Big|=\Big|\sum_{n\in\NN}\xi_n\overline{g_n}\Big|\leq c\norm{}{g}=c\norm{\ell^2}{\{g_n\}_{n\in\NN}},\ \forall g\in\H,
$$
where $g_n:=\ip{g}{e_n}$.  {As $C_e:\H\rightarrow \ell^2(\NN)$ is bijective, $\xi\in \ell^2(\NN)^\ast=\ell^2(\NN)$.
Moreover, since ${\Ker} T_e=\{0\}$ it follows that  $\V_e(\NN)=V_e(\NN)$  and $\norm{\ell^2}{\cdot}=\norm{e}{\cdot}$. }

\subsubsection{Riesz basis}
Now consider a Riesz basis $r=\{r_n\}_{n\in\NN}$. Then $r_n=A e_n$ for some $A\in GL(\H)$  \cite{christ}. Therefore
 {$\V_r(\NN)=V_r(\NN)=\ell^2(\NN)$} as sets, but with equivalent (not necessary equal) norms, since
\begin{align*}
\norm{r}{\xi}&=\sup_{\norm{}{g}=1}\Big|\sum_{n\in\NN}\xi_n\ip{r_n}{g}\Big|=\sup_{\norm{}{g}=1}\Big|\sum_{n\in\NN}\xi_n\ip{e_n}{A^\ast g}\Big|
\\
&=\sup_{\norm{}{g}=1}\norm{}{A^\ast g}\Big|\sum_{n\in\NN}\xi_n\ip{e_n}{\frac{A^\ast g}{\norm{}{A^\ast g}}}\Big|\leq \norm{}{A}\sup_{\norm{}{g}=1}\Big|\sum_{n\in\NN}\xi_n\ip{e_n}{ g}\Big|=\norm{}{A}\norm{\ell^2}{\xi}\, , \;\forall\, \xi\in \ell^2.
\end{align*}
The lower inequality follows by a similar argument.

\subsubsection{Discrete upper and lower-semi frames}
Let $\theta=\{\theta_n\}_{n\in\NN}$ be a discrete frame, $m=\{m_n\}_{n\in\NN}\subset\CN\backslash\{0\}$ and define
$\psi:=\{m_n\theta_n\}_{n\in\NN}$. If $\{|m_n|\}_{n\in\NN}\in c_0$, then $\psi$ is  an upper semi-frame; if $\{|m_n|^{-1}\}_{n\in\NN}\in c_0$, then $\psi$ is  a lower semi-frame. Observe that in both cases $\psi$ is not a frame.

To see this, let $\{|m_n|\}_{n\in\NN}\in c_0$. Then, for every $\varepsilon>0$ there exists $N\in\NN$ such that $|m_n|\leq\varepsilon,\ \forall \,n\geq N$.
Take $f\in \mathrm{span}\{\psi_1,...,\psi_{N-1}\}^\bot$, then
$$
\sum_{n\in\NN}|\ip{f}{\psi_n}|^2=\sum_{n\geq N}|\ip{f}{\psi_n}|^2\leq \varepsilon^2 \sum_{n\in \NN}|\ip{f}{\theta_n}|^2\leq C\varepsilon^2\norm{}{f}^2.
$$
Hence the lower frame inequality cannot be satisfied. The same argument with inverse inequalities yields the result for $\{|m_n|^{-1}\}_{n\in\NN}\in c_0$.

It can easily be seen that $V_{\psi}(\NN)=M_{1/m} (V_\theta(\NN))=M_{1/m} (\ran C_\theta)$ as sets, where $M_m$ is the multiplication operator defined by $(M_m \xi)_n=m_n\xi_n$. Moreover, $\norm{\psi}{\cdot}\asymp\norm{\ell^2_m}{\cdot}$, where $\norm{\ell^2_m}{\xi}:=\sum_{n\in\NN}|\xi_n m_n|^2$.

 {
Now we will apply Theorem \ref{theo-partner} to show that there exists $\psi$ such that $(\psi,\phi)$ is a reproducing pair.}
  {
We first identify $\widehat T_\phi$. Let $\xi\in V_\phi(\nN)$, then
\begin{equation*}
 \widehat T_\phi \xi=\sum_{n\in\nN}\xi_n \phi_n=\sum_{n\in\nN}\xi_n m_n \theta_n=T_\theta(M_m\xi).
\end{equation*}
The identification $V_{\psi}(\NN)=M_{1/m} (\ran C_\theta)$ immediately implies that $\ran \T_\psi=\H$.
 In order to check condition \eqref{second-assumption} we observe that
 the reproducing kernel property yields}
\begin{equation*} \widehat T_\phi^{-1}f=M_{1/m}C_\theta S_\theta^{-1}f , \; \forall\, f\in\H.
\end{equation*}
Hence, for every fixed $k\in\nN$, we have
$$
\sum_{n\in\nN}|m_k^{-1}\ip{S_\theta^{-1}e_n}{\theta_k}|^2=|m_k|^{-2}\norm{}{S_\theta^{-1}\theta_k}^2<\infty .
$$
One natural choice of a reproducing partner is
$\psi:=\{(1/\overline{m_n})\theta_n\}_{n\in\NN}$  as $S_{\psi,\phi}=S_\theta\in GL(\H)$.

\subsubsection{Gabor systems}
 Let $a,b>0$ and $g\in L^2(\RN)$, the Gabor system $\G(g,a,b)$ is given by
$$
\G(g,a,b):=\{T_{an}M_{bm}\varphi\}_{n,m\in\ZN},
$$
where $T_x$ denotes the translation and $M_\omega$ the modulation operator. For an overview on Gabor analysis, see \cite{groe01}.

 Reproducing pairs appear to be a promising approach for the study of Gabor systems at critical density ($a\cdot b=1$) since
 the well-known Balian-Low theorem (BLT) states that  if $g$ is well-localized in both time and frequency, then
$\G(g,a,1/a)$ is not a frame.

We expect that it is possible to construct a reproducing pair consisting of two Gabor systems where one window beats the obstructions of BLT.

When Gabor first introduced these systems in \cite{gab46}, he considered the family $\G(\varphi,1,1),$
where $\varphi(t):=2^{1/4}e^{-\pi t^2}$, i.e., a system of integer time-frequency shifts of the Gaussian.
There is no  Gabor system with a window in $L^2(\RN)$ which is dual to $\G(\varphi,1,1)$. However, Bastiaans \cite{ba80} and Janssen \cite{janss} have shown that there is $\gamma\notin L^2(\RN)$, such that $\G(\gamma,1,1)$ is dual  in a weak distributional sense.

 The question if there is an arbitrary reproducing partner for $\G$ is unsolved. Theorem \ref{theo-partner} provides a helpful tool for further research in this direction.

\subsection{Continuous examples}

\subsubsection{Continuous frames}

If $\phi$ is a continuous frame,   Corollary  \ref{cor-Hil} implies that $V_\phi(X,\mu)\subseteq L^2(X,\mu)$.
Now, since $L^2(X,\mu)=\ran C_\phi\oplus \Ker D_\phi$, it follows that $V_{\phi}(X,\mu)[\norm{\phi}{\cdot}]\simeq \ran C_\phi[\norm{L^2}{\cdot}]$.

Observe that there may exist $\xi\in\V_\phi(X,\mu)$, such that $\xi\notin L^2(X,\mu)$. In particular, if there exists a lower semi-frame $\psi$ which is not Bessel such that
$(\psi,\phi)$ is a reproducing pair, then $\ran C_\psi\subset \V_\phi(X,\mu)$. See \cite[Section 4]{speck-bal} for an example.
Nevertheless, there is always a unique $f\in\H$ such that
$\xi=C_\phi f+\xi_0$, where $[\xi]_\phi=[C_\phi f]_\phi$ and $\xi_0 \in \Ker T_\phi$, i.e. $\xi_0\notin L^2(X,\mu)$.

\subsubsection{1D continuous wavelets}

Let $\phi,\psi\in L^2(\RN,\ud x)$ and consider the continuous wavelet systems $ \phi_{x,a}=T_xD_a\phi, $
where, as usual, $T_x$ denotes the translation and $D_a$ the dilation operator.
 {If
\be
\int_\RN  |\widehat\psi(\omega) \widehat\phi(\omega)|\frac{\ud\omega}{|\omega|}<\infty
\label{eq-admiss}
\en
then $(\psi,\phi)$ is a reproducing pair for $L^2(\RN,dx)$ with $S_{\psi,\phi}=c_{\psi,\phi}I$
  \cite[Theorem 10.1]{groe01}, where
$$
 c_{\psi,\phi} :=\int_\RN \overline{\widehat\psi(\omega)}\widehat\phi(\omega)\frac{\ud\omega}{|\omega|}.
$$}
 Actually this is just another way of expressing the well-known orthogonality relations of wavelet transforms --- or, for that matter, of all coherent states associated to square integrable group representations
  \cite[Chaps. 8 and 12]{CSbook}.
 {For $\psi=\phi$, the cross-admissibility condition \eqref{eq-admiss} reduces to the classical admissibility condition
\be
c_\phi:=\int_\RN |\widehat\phi(\omega)|^2\frac{\ud\omega}{|\omega|}<\infty.
\label{class-admiss}\en}
\\
 {Considering the obvious inequalities
$$
| c_{\psi,\phi}| \leq \int_\RN  |\widehat\psi(\omega) \widehat\phi(\omega)|\frac{\ud\omega}{|\omega|}
\leq c_\phi^{1/2}c_\psi^{1/2},
$$
we see that condition \eqref{eq-admiss} is automatically satisfied whenever $\phi$ and $\psi$ are both admissible.}
However,  it is possible to choose a mother wavelet
 $\phi$  that does not satisfy the   admissibility condition   \eqref{class-admiss}
 and still obtain a reproducing pair $(\psi,\phi)$.

Consider for example the Gaussian window $\phi(x)=e^{-\pi x^2}$, then $c_\phi=\infty$ which implies that $\phi$ is not a
continuous wavelet frame. However, if one defines $\psi\in L^2(\RN,dx)$ in the Fourier domain via $\widehat\psi(\omega)=|\omega| \widehat\phi(\omega)$,
it follows that $ 0<c_{\psi,\phi}=\|\phi\|_2^2<\infty$.
Thus we
conclude that $(\psi,\phi)$ is a reproducing pair.}

Needless to say, the same considerations apply to $D$-dimensional continuous wavelets \cite{CSbook}.

\medskip

\subsubsection{A continuous upper semi-frame: affine coherent states}
\label{subsub-upperframe}

 In \cite[Section 2.6]{ant-bal-semiframes1} the following example of an upper semi-frame is investigated.
 Define $\H_n:=L^2(\RN^+,r^{n-1}\ud r)$, where $n\in\NN$ and the following measure space $(X,\mu)=(\RN,\ud x)$.
 Let $\psi\in \H_n$ and define  {the affine coherent state}
$$
\psi_x(r)=e^{-ixr}\psi(r),\ \ r\in\RN^+.
$$
Then $\psi$ is admissible if
$\sup_{r \in {\RN}^{+}}{\mathfrak s}(r)  = 1,$  where   ${\mathfrak s}(r):=2\pi r^{n-1}|\psi (r)|^{2}$,
and $|\psi(r)|\neq 0$, for a.e. $r\in\RN^+$.
The frame operator is  given by the multiplication operator on $\H_n$
$$
(S_\psi f)(r)= {\mathfrak s}(r) f(r),
$$
and, more generally,
$$
(S_\psi^m f)(r)= [{\mathfrak s}(r)]^m f(r), \; \forall\, m\in\ZN.
$$
Hence $S_\psi$ is bounded  and $S_\psi^{-1}$ is unbounded.

First we  identify $\Ker {\sf D}_\psi$ as the space ${\mathcal K}_+:=\{\eta\in L^2(\RN): \widehat{\eta}(\omega)=0,\ \mbox{for a.e.}\ \omega\geq 0\}$. For every
 $\xi\in L^2(\RN)$ and $g \in \H_n$, we have, indeed, the following equality
$$
\ip{{\sf D}_\psi \xi}{g}=\int_{\RN^+}\left(\int_\RN\xi(x)e^{-ixr}\psi(r) \ud x\right)\ov{g(r)}r^{n-1}\ud r=\int_{\RN^+} \widehat\xi(r)\psi(r)\overline{g(r)}r^{n-1}\ud r,$$
which easily implies that $\Ker {\sf D}_\psi={\mathcal K}_+$.

 Thus in this case we find that
 $\Ker {\sf D}_\psi = (\ran {\sf C}_\phi )^\bot = \mathcal{K}_+\neq\{0\}$ (it is infinite dimensional), an
example of the situation described in   {Section \ref{sec-nondeg}.}
\medskip

The function $\psi$ enjoys the interesting property that we can characterize the space $V_\psi(\RN,\ud x)$
and its norm.
First, we show that
 $\xi\in V_\psi(\RN,\ud x)$ implies $\widehat\xi \psi \in \H_n$ and $\norm{\psi}{\xi}= \norm{}{\widehat\xi \psi}$.
Indeed, let $\xi\in V_\psi(\RN,\ud x)$ and $\psi,g\in \H_n$. Then we have,
{
\begin{equation}\label{eq-51}
\begin{aligned}
\ip{\xi}{C_\psi g}&=\ip{\widehat T_\psi \xi}{g}
=\int_{\RN^+} \int_\RN \xi(x)e^{-ixr}\psi(r)\ud x\ \overline{g(r)}r^{n-1}\ud r
\\
&=\int_{\RN^+}\widehat\xi(r)\psi(r)\overline{g(r)}
r^{n-1}\ud r =\ip{\widehat\xi \psi}{g}.
\end{aligned}
\end{equation}
Hence, $T_\phi\xi=\widehat\xi \psi$ which in turn implies that $\widehat\xi$ has to be given by an almost everywhere defined
function which satisfies $\widehat\xi \psi\in\H_n$.
Moreover, \eqref{eq-51} yields
\begin{equation}\label{eq-59}
\norm{\psi}{\xi}=\sup_{\norm{}{g}\leq 1}|\ip{\xi}{{\sf C}_\psi g}|=\sup_{\norm{}{g}\leq 1}|\ip{\widehat\xi \psi}{g}|=  \norm{}{\widehat\xi \psi}.
\end{equation}
Then again, by the same reasoning, the previous chain of equalities shows that a measurable function $\xi$ is contained in
$V_\psi(\RN,\ud x)$ provided that $\xi \in \mathcal{F}^{-1}(\psi^{-1}\H_n)$.
\beprop
{Let $\psi\in \H_n$, then, as sets,
\begin{equation*}
V_\psi(\RN,\ud x)=\left\{\xi:X\rightarrow\CN \mbox{ measurable}:\ \xi\in\mathcal{F}^{-1}(\psi^{-1}\H_n)\right\}/\Ker T_\phi
\end{equation*}
}
and $\norm{\psi}{\xi}=\norm{}{\widehat\xi\psi}\, , \forall \,\xi\in  V_\psi(\RN,\ud x)$.
\enprop

The inverse Fourier transform is taken in the sense  of distributions, if needed.

\medskip
 In the quest of a reproducing partner for $\psi$ we will first treat the question if there exists an affine coherent
state $\phi_x(r)=e^{-ixr}\phi(r),\  r\in\RN^+,\  \phi\in\H_n$, such that $(\psi, \phi)$ forms a
reproducing pair. Indeed, since $\psi$ is Bessel and not a frame, its dual $\phi$ is by necessity a lower semi-frame,
whereas an affine coherent state must be Bessel,   but can never satisfy the lower frame bound.
Hence, there is no pair of affine coherent states forming a reproducing pair.  This fact can also be proven  by an explicit calculation.

Finally, we have here an example of the situation described in Remark \ref{rem217}, namely,
$C_\psi$ being an isometry by Corollary \ref{isomcor}, but
$\ran \widehat T_\psi \neq \H$.  {We have already seen in \eqref{eq-51} that
$
\widehat T_\psi \xi=\widehat\xi\psi.
$
If $\ran \widehat T_\psi = \H$, an arbitrary element  $h\in\H_n = L^2(\RN^+,r^{n-1}\ud r)$ may be written as
$h= \widehat T_\psi \xi = \widehat \xi\psi $ for some
$\xi\in V_\psi(\RN,\ud x)$.}
This applies, in particular, to $\psi$ itself, which also belongs to $\H_n$. This in turn implies that there exists $\xi$,
such that $\widehat\xi(r)=1$ for a.e. $r\geq0$. But there is no
 function that satisfies this condition (however the $\delta$-distribution does the job).

This has two major consequences. First, it shows that $V_\psi(\RN,\ud x)$ is \emph{not} a Hilbert
space, since it is not complete.
Second, there is \emph{no} reproducing partner for $\psi$ making it a reproducing pair.

\subsubsection{Continuous wavelets on the sphere}
\label{subsub-wavsphere}

Next we consider the continuous wavelet transform on the 2-sphere $\mathbb{S}^2$   \cite{CSbook,jpavand99}.
  For a mother wavelet $\phi\in \H=L^2(\mathbb{S}^2,\ud\mu)$, define
  $$\phi_{x,a}:=R_x D_a\phi, \mbox{ where }(x,a)\in X:= SO(3)\times \RN^+.$$ Here, $D_a$ denotes the stereographic dilation operator
  and $R_x$ the unitary rotation on $\mathbb{S}^2$.

It has been shown in \cite[Theorem 3.3]{jpavand99} that the operator $S_\phi $ is given by a Fourier multiplier
  $\widehat{S_\phi f}(l,n)=s_\phi(l)\hat f(l,n)$ with the symbol $s_\phi$ given by
  $$
  s_\phi(l):=\frac{8\pi^2}{2l+1}\sum_{|n|\leq l}\int_0^\infty\big|\widehat{D_a\phi}(l,n)\big|^2\frac{\ud a}{a^3},\ l\in\{0\}\cup\nN.
  $$
  {If  ${\sf m}\leq s_\phi(l)<\infty$ for all $l\in\{0\}\cup\nN$, it follows that $\phi$ is a
 lower semi-frame and $S_\phi$ is densely defined.}

We will apply Theorem \ref{theo-partner} to investigate the existence of a reproducing partner for $\phi$.
 {First, we show that $\ran \T_\phi=\H$.
The operator $M_\phi^{-1}$ defined by $\widehat{M_\phi f}(l,n)=s_\phi(l)^{-1}\widehat f(l,n)$ is bounded and constitutes a right
inverse to $S_\phi$. Hence, for every $f\in \H$, it holds
$$
f=S_\phi M_\phi f=\T_\phi [C_\phi M_\phi f]_\phi\in\ran\T_\phi.
$$
}
  The spherical harmonics $Y_l^n$ form an orthonormal basis of $L^2(\mathbb{S}^2,\ud\mu)$. Choosing
  $\xi_{l,n}(x,a):= C_\phi (S_\phi^{-1} Y_l^n)(x,a)$ as a representative of
  $[\T_\phi^{-1}Y_l^n]_\phi$  yields for every $(x,a)\in\RN\times\RN^+$:
\begin{align*}
  \sum_{l=0}^\infty \sum_{|n|\leq l}|\xi_{l,n}(x,a)|^2 &=  \sum_{l=0}^\infty \sum_{|n|\leq l}|C_\phi (S_\phi^{-1} Y_l^n)(x,a)|^2
 = \sum_{l=0}^\infty \sum_{|n|\leq l}|\ip{S_\phi^{-1}Y_l^n}{\phi_{x,a}}|^2
 \\
 &=   \sum_{l=0}^\infty \sum_{|n|\leq l}|\widehat{S_\phi^{-1}\phi_{x,a}}(l,n)|^2
  =\sum_{l=0}^\infty \sum_{|n|\leq l}|s_\phi(l)^{-1}\widehat\phi_{x,a}(l,n)|^2
 \\
 & \leq
 \frac{1}{m}\sum_{l=0}^\infty \sum_{|n|\leq l}|\widehat\phi_{x,a}(l,n)|^2=\frac{1}{m}\norm{}{\phi_{x,a}}^2<\infty.
\end{align*}

Moreover, as for the wavelets on $\RN^d$, it is  possible to choose
 another continuous wavelet system $\psi_{x,a}$ as reproducing partner if the symbol $s_{\psi,\phi}$, defined by
     $$
  s_{\psi,\phi}(l):=\frac{8\pi^2}{2l+1}\sum_{|n|\leq l}\int_0^\infty\widehat{D_a\psi}(l,n)\overline{\widehat{D_a\phi}(l,n)}
  \frac{da}{a^3}.
  $$
  satisfies $m\leq|s_{\psi,\phi}(l)|\leq M$ for all $l\in \{0\}\cup\nN$.

\section{Outcome}

We have seen that the notion of reproducing pair is quite rich.  It generates a whole mathematical structure.
We have given several concrete examples in Section \ref{sec-examples}. These, and additional ones, should  allow one
to better specify the best assumptions to be made on the measurable functions or, more precisely, on the nature of the range of the
analysis operators  {$C_\psi, C_\phi$.}
Let $(\psi, \phi)$ be a reproducing pair. By definition,
\be\label{eq-defS}
\ip{S_{\psi,\phi}f}{g}=\int_X \ip{f}{\psi_x}\ip{\phi_x}{g}\ud\mu(x) =  \int_X C_\psi f (x) \; \ov{C_\phi g(x)}  \ud\mu(x)
\en
is well defined for all $f,g\in\H$. The r.h.s. is the $L^2$ inner product, but generalized, since in general $C_\psi f ,C_\phi $
need not belong to   $L^2(X,\ud\mu)$. Thus clearly the analysis should be made in the context of \pip s \cite{pip-book}.
This is a topic for future research.

Another interesting direction  consists in considering a whole family of {$\mu$-total}, weakly measurable
functions $\phi : X \to \H$, instead of  only one.
To each $\phi\in\G$ we can associate the pre-Hilbert space $V_\phi(X,\mu)[\norm{\phi}{\cdot}]$ and take its completion
$\widetilde{V_\phi}(X,\mu)[\norm{\phi}{\cdot}]$  . If $\phi$ has a partner $\psi\in \G$ such that $(\psi,\phi)$ is a
reproducing pair, both spaces $V_\phi(X,\mu)=\widetilde{V_\phi}(X,\mu)[\norm{\phi}{\cdot}]$ and
$V_\psi(X,\mu)=\widetilde{V_\psi}(X,\mu)[\norm{\phi}{\cdot}]$ are Hilbert spaces, conjugate dual to each other. In the general case,
however, the question of completeness of $V_\phi(X,\mu)[\norm{\phi}{\cdot}]$ is open. Can one find conditions under which it
is true?

 {
\section*{Acknowledgement}
This work was partly supported by the Austrian Science Fund (FWF) through the START-project
FLAME (Frames and Linear Operators for Acoustical Modeling and Parameter Estimation): Y 551-N13 and by the Istituto Nazionale di Alta Matematica (GNAMPA project ``Propriet\`a spettrali di quasi *-algebre di operatori"). JPA acknowledges gratefully
the hospitality of the Acoustic Research Institute, Austrian Academy of Science, Vienna, and that of the Dipartimento di Matematica e Informatica, 
Universit\`a di Palermo, whereas MS and CT
acknowledge that of the Institut de Recherche en Math\'ematique et  Physique, Universit\'e catholique de Louvain.
}


\begin{thebibliography}{99}
\bibitem{foot1}As usual, we identify a function $\xi$ with its residue class in  $L^{2}(X, \ud\mu)$.
\end{thebibliography}
\end{document}